\documentclass[a4paper, 12pt, oneside]{amsart}

\usepackage[top=35mm,bottom=15mm,left=20mm,right=20mm]{geometry}
\usepackage{mathrsfs}
\usepackage{amsfonts}
\usepackage{amssymb}
\usepackage{amsxtra}
\usepackage{color}
\usepackage{dsfont}
\usepackage{bbm}

\usepackage[english,polish]{babel}
\usepackage[colorlinks,citecolor=blue,urlcolor=blue,bookmarks=true]{hyperref}

\newcommand{\al}{\alpha}

\newcommand{\N}{\mathbb{N}}
\newcommand{\R}{\mathbb{R}}

\newcommand{\Per}{\mathrm{Per}}

\newcommand{\calL}{\mathcal{L}}
\newcommand{\calM}{\mathcal{M}}
\newcommand{\calC}{\mathcal{C}}
\newcommand{\ud}{\mathrm{d}}

\newcommand{\g}{g_\Omega}

\newcommand{\Rd}{\mathbb{R}^d}
\newcommand{\eps}{\varepsilon}
\newcommand{\sgn}{\operatorname{sgn}}
\newcommand{\norm}[1]{{\lVert #1 \rVert}}

\newcommand{\ind}{\mathds{1}}

\newcommand{\vertiii}[1]{{\left\vert\kern-0.25ex\left\vert\kern-0.25ex\left\vert #1 
    \right\vert\kern-0.25ex\right\vert\kern-0.25ex\right\vert}}

\newcommand{\WUSC}[3]{\textrm{\rm WUSC}\left(#1,#2,#3\right)}
\newcommand{\uC}{{\overline{C}}}

\def \Aa{\rm{(A1)}}
\def \Ab{\rm{(A2)}}
\def \Ac{\rm{(A3)}}

\newtheorem{theorem}{Theorem}
\newtheorem{proposition}{Proposition}
\newtheorem{lemma}{Lemma}
\newtheorem{corollary}{Corollary}

\theoremstyle{definition}
\newtheorem{example}{Example}

\title{Asymptotic expansion of the nonlocal heat content}
\author[T.~Grzywny]{Tomasz Grzywny}
\address{Faculty of Pure and Applied Mathematics, Wroc\l{}aw University of Science and Technology, Wyb. Wyspia\'nskiego 27, 50-370 Wroc\l{}aw, Poland.}
\email{tomasz.grzywny@pwr.edu.pl}
\author[J.~Lenczewska]{Julia Lenczewska}
\address{Faculty of Pure and Applied Mathematics, Wroc\l{}aw University of Science and Technology, Wyb. Wyspia\'nskiego 27, 50-370 Wroc\l{}aw, Poland.}
\email{julia.lenczewska@pwr.edu.pl}
\thanks{This research was partially supported by National Science Centre (Poland) grant 2019/33/B/ST1/02494.}

\subjclass[2020]{60G51, 60G52, 60J76, 35K05}
\keywords{asymptotic expansion, characteristic exponent, convolution semigroup, fractional Laplacian, heat content, H\"older space, 
	 L\'{e}vy measure, nonlocal operator, perimeter, regular variation}
\begin{document}
\selectlanguage{english}

\begin{abstract}
Let $(p_t)_{t\geq0}$ be a convolution semigroup of probability measures on $\Rd$ defined by
$$\int_{\Rd} e^{i\left<\xi,x\right>} p_t(\ud x)=e^{-t\psi(\xi)}\,, \quad \xi \in\Rd,$$ and let $\Omega$ be an open subset of $\mathbb{R}^d$ with finite Lebesgue measure. In this article we consider the quantity
$H_{\Omega}(t)= \int_{\Omega}\int_{\Omega-x}p_t( \ud y)\ud x$,
which is called the heat content. We study its asymptotic expansion under mild assumptions on $\psi$, in particular in the case of the $\al$-stable semigroup.
\end{abstract}

\maketitle
\section{Introduction}
Let $d\in\N$. We consider a semigroup of probability measures $(p_t)_{t\geq0}$ given by
 $$
\int_{\Rd} e^{i\langle\xi,x\rangle} p_t(\ud x)
=e^{-t\psi(\xi)}\,,\qquad \xi \in\Rd\,,
$$
where $\psi$ is a symbol defined by
\begin{align*}
\psi (\xi) =  \int_{\Rd}\left( 1-e^{i \langle\xi,z\rangle}\right)\nu (\ud z),\quad \xi\in\Rd,
\end{align*}
and $\nu(\ud z)$ is a Borel measure satisfying
$$
\nu(\{0\})=0\,,
\qquad\quad
\int_{\Rd} (1\land |z|) \, \nu(\ud z)<\infty \,.
$$
Let $\{P_t\}_{t\geq0}$ be the convolution semigroup of operators on $\calC_0(\Rd)$ defined by $(p_t)_{t\geq 0}$ and let $\calL$ denote its infinitesimal generator, which for $f\in C_c^2(\Rd)$ is given by the formula
\begin{align}
\calL f (x) &=  \int_{\Rd} \left(f(x+z) - f(x)\right) \nu (\ud z). \label{L}
\end{align}

Let $\Omega$ be a non-empty, open subset of $\Rd$ such that its Lebesgue measure $|\Omega|$ is finite. We consider the following quantity associated with the semigroup $(p_t)_{t\geq0}$,
\begin{align*}
H _{\Omega} (t) =   \int_{\Omega}\int_{\Omega-x}p_t( \ud y)\ud x,
\end{align*}
which we will call \textit{heat content}.

We note that the function $u(t,x) = \int_{\Omega-x}p_t(\ud y)$ is the weak solution of the initial value problem
\begin{align*}
\frac{\partial}{\partial t}u(t,x) &= \mathcal{L}\, u(t,x),\quad t>0,\, x\in \Rd, \\
u(0,x) &= \mathds{1}_{\Omega}(x). 
\end{align*}
Therefore, the quantity $H_\Omega (t)$ can be interpreted as the amount of \textit{heat} in $\Omega$ if its initial temperature is one whereas the initial temperature of $\Omega^c$ is zero.

Our main goal is to study the
asymptotic expansion
 of $H_{\Omega}(t)$ 
for small $t$. 
We observe that
$$
H_{\Omega}(t) = |\Omega| - H(t),
$$
where 
$$
H(t) = \int_{\Omega} \int_{\Omega^c-x} p_t(\ud y) \ud x,
$$
and hence it suffices to work with the function $H(t)$. One of the main results of \cite{MR3563041} states that, for small $t$,
$$
H_{\Omega}(t) = |\Omega| -t\Per_{\nu}(\Omega)+o(t),$$
where
 $\Per_{\nu}(\Omega)$ is the nonlocal perimeter related to the measure $\nu$, defined as
\begin{align}\label{X_perimeter}
\Per_{\nu}(\Omega)= \int_{\Omega}\int_{\Omega ^c-x}\nu (\ud y)\, \ud x .
\end{align}
For instance, if $\nu$ is the $\al$-stable L{\'e}vy measure with $\al \in (0,1)$, denoted by
$\nu^{(\al)}(\ud z) =  \mathcal{A}_{d,-\alpha}|z|^{-d-\al} \, \ud z$, 
where 
 $$
\mathcal{A}_{d,-\alpha} = \frac{2^{\alpha}\Gamma\left(\frac{d+\alpha}{2}\right)}{\pi^{d/2} \left|\Gamma\left(-\frac{\alpha}{2}\right)\right|},
$$
then $\Per_{\nu^{(\al)}}(\Omega) = \mathcal{A}_{d,-\al} \Per_{(\al)}(\Omega)$, with 
$\Per_{(\al)}(\Omega)$ being the well-known $\al$-perimeter \cite{MR2675483}, given for $0<\al <1$ by
\begin{align*}
\Per_{(\al)} (\Omega)= \int_{\Omega}\int_{\Omega ^c} \frac{\ud y\, \ud x}{|x-y|^{d+\al}}.
\end{align*}
In the present paper, we shall establish the next terms of the asymptotic expansion of the heat content related to convolution semigroups. Such result is new even for the fractional Laplacian $(-\Delta)^{\alpha/2}$ (in our setting we consider $\alpha\in(0,1)$). For instance, if $1/\alpha$ is a natural number, we prove the following expansion of the heat content for the fractional Laplacian:
$$H_\Omega(t)=|\Omega|+  \sum_{n=1}^{1/\alpha-1} \frac{(-1)^{n}}{n!} t^n \Per_{\nu^{(n\alpha)}} (\Omega) +
\frac{(-1)^{1/\alpha}}{(1/\alpha-1)!\pi}t^{1/\alpha}\log(1/t) \Per(\Omega) + o(t^{1/\al} \log(1/t)) ,$$
where $\Per$ is the classical perimeter of the set, see \eqref{Perimeter_def}.  
A natural question arises: will the next term be the mean curvature or its non-local counterpart?

The key observation to obtain the asymptotic expansion is that the heat content can be expressed as the action of the semigroup on the {\it covariance function} of a set. We give more general results, concerning the asymptotic expansion of $P_tf$ for functions $f$ belonging to H{\"o}lder space. Our standing assumption is the {\it weak upper scaling} of the
symbol $\psi^*$, or equivalently, certain scaling properties of the {\it concentration function} of a L\'evy measure,  see Theorem \ref{thm:thm1}. For a class of convolution semigroups, for instance for the semigroup associated to $\log(1+\Delta)$, we get the full expansion, see Theorem \ref{thm:thm2}. We apply these results to obtain the expansion of heat content, which are stated in Corollaries \ref{thm:cor1} and \ref{thm:cor2}.  
Using asymptotic expansion of the heat kernel of the fractional Laplacian, we give more explicit asymptotic expansion in the case of $\al$-stable semigroups, see Theorems \ref{thm3} and \ref{thm4}.

Heat content related to the Gaussian semigroup ($\calL = \frac{1}{2} \Delta$) of a set at time $t$ was defined by van den Berg \cite{MR3116054} by means of the heat semigroup. Van den Berg and Gilkey \cite{MR1262245} proved that the heat content, regarded as a function of variable $t$, has an asymptotic expansion as $t$ tends to $0$. The first three terms in the expansion include the volume of the set, its perimeter and its mean curvature. The short time behavior of heat semigroup in connection
with the geometry of sets with finite perimeter was also studied by Angiuli, Massari and Miranda \cite{MR3019137}. The concept of the heat content was extended to the nonlocal setting of $\al$-stable semigroups in 2016 by Acu\~{n}a Valverde \cite{MR3606559}, who described the small-time asymptotic behavior of the nonlocal heat content in this case. In the one-dimensional case, the number of terms of the expansion depends on the parameter $\al$, and in the multidimensional case, there are the first two terms of the expansion. The same author found first three terms of the asymptotic expansion for the Poisson heat content over the unit ball \cite{MR4224349} and over convex bodies \cite{MR4158754}. 
In 2017, Cygan and Grzywny \cite{MR3563041} introduced the notion of a nonlocal heat content related to general probabilistic convolution semigroups and generalized the mentioned results of Acu\~{n}a Valverde. Later, they proved similar results for the generalized heat content related to convolution semigroups \cite{MR3859849}. Maz\`{o}n, Rossi and Toledo \cite{MR3930619} found the full asymptotic expansion of the heat content for nonlocal diffusion with nonsingular kernels. Recently, in a more general setting, the heat
content related to fractional Laplacian in Carnot groups was studied by Ferrari, Miranda, Pallara, Pinamonti, and Sire \cite{MR3732178}.

\section{Preliminaries}
\subsection{Convolution semigroups}\label{sec21}
For $f: \Rd \to \R$, let
$$
P_t f (x) = \int_{\Rd} f(x+y) \, p_t(\ud y), \quad \, t \geq 0\,, \, x \in \Rd\,.
$$
The generator $\calL$ of the semigroup $\{P_t\}_{t\geq 0}$ is defined as
$$
\mathcal{L}f(x) = \lim_{t \to 0^+} \frac{P_t f(x) - f(x)}{t},
$$
for functions for which the above limit exists.

We denote by $\calC_0(\Rd)$ the space of continuous functions $f: \Rd \to \R$ vanishing at infinity. For $\beta\in(0,1]$ we define 
$$
\vertiii{f}_{\beta} := \sup_{|x-y|\leq 1} \frac{|f(x)-f(y)|}{|x-y|^{\beta}}.
$$
We will consider the H\"older space
$$
\calC_0^{\beta} = \left\{ f \in \calC_0(\Rd): \|f\|_{\beta}:= \vertiii{f}_{\beta} + \|f\|_{\infty} < \infty \right\}.
$$
\cite[Theorem 3.2]{MR3926121} implies that, for a fixed $\beta \in (0,1]$, if $\int_{|y|<1} |y|^{\beta} \nu (\ud y) < \infty$, then for $f \in \calC_0^{\beta}$,
\begin{equation}\label{eq:LH}
\calL f(x) = \int_{\Rd} (f(x+y)-f(x)) \, \nu (\ud y).
\end{equation}

The real part of the symbol $\psi$
equals
$
{\rm Re}[\psi(\xi)]=\int_{\Rd}\big( 1-\cos \left<\xi,z\right> \big) \, \nu(\ud z)
$. We will consider its radial, continuous and non-decreasing majorant
defined by
$$
\psi^*(r)=\sup_{|\xi|\leq r} {\rm Re}[\psi(\xi)],\qquad r>0\,.
$$

For $r>0$ we define the {\it concentration function}
$$
h(r)= \int_{\Rd}\left(1\wedge\frac{|x|^2}{r^2}\right)\nu(\ud x)\,.
$$
By \cite[Lemma~4]{MR3225805}, for all $r>0$,
\begin{align}\label{ineq:comp_TJ}
\frac{1}{8(1+2d)} h(1/r)\leq \psi^*(r) \leq 2 h(1/r)\,.
\end{align}
Hence $h$ is a more tractable version 
of $\psi^*$. By \cite[Lemma 2.1]{MR4140542},
\begin{equation}\label{eq:h_estimate}
\int_{|z|\geq r }  \nu(\ud z)\leq  h(r)  \quad \textnormal{for all} \,\, r>0.
\end{equation}
By \cite[Lemma 2.7]{MR4140542}, if $f\colon [0,\infty)\to [0,\infty)$ is differentiable, $f(0)=0$, $f'\geq 0$ and $f'\in L^1_{loc}([0,\infty))$, then for all $r>0$,
\begin{equation}\label{eq:lem26}
\int_{|z|<r} f(|z|) \,\nu(\ud z) =\int_0^r f'(s) \nu (|x|\geq s)\, \ud s - f(r) \nu(|x|\geq r).
\end{equation}

Let $\theta_0 \geq 0$ and $\phi : (\theta_0, \infty) \to [0,\infty]$. We say that
$\phi$ satisfies {the} {\it weak upper scaling condition} (at infinity) if there are numbers
$\al \in \R$, 
and $\uC \in [1,\infty)$ such that
\begin{equation}\label{eq:USC}
 \phi(\lambda\theta)\le
\uC \lambda^{\al} \phi(\theta) \quad \mbox{for}\quad \lambda\ge 1, \quad\theta
>\theta_0
\end{equation}
In short, $\phi\in$ $\WUSC{\al}{\theta_0}{\uC}$. This condition will be our standing assumption on the symbol $\psi^*$ throughout the paper.

The following auxiliary result is a consequence of \eqref{ineq:comp_TJ}, \eqref{eq:h_estimate} and \eqref{eq:lem26}.
\begin{lemma}\label{lem:equiv}
Let $\al \in(0,2]$, $\uC \in [1,\infty)$ and $\theta_0 \in[0,\infty)$. Consider
\begin{enumerate}
\item[\Aa] $\psi^* \in \WUSC{\al}{\theta_0}{\uC}$
\item[\Ab] There is $C>0$ such that for all $\lambda \leq 1$ and $r<1/\theta_0$,
$$
h(\lambda r)\leq C \lambda^{-\al}h(r).
$$
\end{enumerate}
Then, $\Aa$ implies $\Ab$ with $C = c_d \uC$, where $c_d=16(1+2d)$, while $\Ab$ gives $\Aa$ with $\uC=c_d C$. If additionally $\theta_0 \in [0,1)$, then $\Aa$ and $\Ab$ each imply 
\begin{enumerate}
\item[\Ac] For all $\eps>0$,
$$
\int_{|y|<1} |y|^{\al+\eps} \nu(\ud y) < \infty.
$$
\end{enumerate}
\end{lemma}
\begin{proof}
We will show that \Aa \, implies \Ab. Using \eqref{ineq:comp_TJ}, \eqref{eq:USC} and again \eqref{ineq:comp_TJ}, we obtain
\begin{align*}
h(\lambda r) \leq 8(1+2d) \psi^*((\lambda r)^{-1}) & \leq 8(1+2d) \uC \lambda^{-\al} \psi^*(r^{-1}) \\
&\leq 16(1+2d) \uC \lambda^{-\al} h(r).
\end{align*}
The converse implication can be proved analogously. It remains to show that \Aa \, implies \Ac. By \eqref{eq:lem26} with $f(s)= s^{\al+\eps}$ and $r=1$,
\begin{align*}
\int_{|y|<1} |y|^{\al + \eps} \, \nu (\ud y) &= (\al+\eps) \int_0^1 s^{\al+\eps-1} \nu(|x|\geq s) \, \ud s - \nu(|x|\geq 1) \\
&\leq (\al+\eps) \int_0^1 s^{\al+\eps-1} h(s) \, \ud s \\
&\leq (\al+\eps) \, C h(1) \int_0^1 s^{\eps-1}  \, \ud s 
= \frac{C(\al+\eps)}{\eps} h(1) <\infty.
\end{align*}
\end{proof}

\subsection{Heat content}
Following \cite[Section 3.3]{MR1857292}, for any measurable set $\Omega \subset \Rd$ we define its perimeter $\Per(\Omega)$ as 
\begin{align}\label{Perimeter_def}
\Per(\Omega) = \sup \left\{ \int_{\Rd}\ind_{\Omega}(x)\mathrm{div}\, \phi (x)\, \ud x:\, \phi \in C_c^1(\Rd,\Rd),\, \norm{\phi}_{\infty}\leq 1 \right\}.
\end{align}
We say that $\Omega$ is of finite perimeter if $\Per(\Omega)<\infty$.
We mention that, by \cite[Proposition 3.62]{MR1857292}, for any open $\Omega$ with Lipschitz boundary $\partial \Omega$ and finite Hausdorff measure $\sigma (\partial \Omega)$ we have 
\begin{align*}
\Per(\Omega) = \sigma (\partial \Omega).
\end{align*}

For any $\Omega \subset\Rd$ with finite Lebesgue measure $|\Omega|$, we define the covariance function $g_\Omega$ of $\Omega$ as follows
\begin{align}\label{g_omega_defn}
g_\Omega (y)=|\Omega\cap (\Omega + y)|=\int_{\Rd}\,\ind_{\Omega}(x)\,\ind_{\Omega}(x-y) \ud x,\quad y\in \Rd.
\end{align}
We recall some important properties of the function $\g$.
Let $\Omega \subset \Rd$ have finite measure. Then, by \cite[Proposition 2, Theorem 13 and Theorem 14]{MR2816305}, 
$\g$ is symmetric, nonnegative, bounded from above by $|\Omega|$,
and
$g_{\Omega} \in \calC_0(\Rd)$.
Moreover, if $\Per(\Omega)<\infty$, then
$g_\Omega$ is Lipschitz with
\begin{equation}\label{prop_iv}
2\norm{g_\Omega}_{\mathrm{Lip}} \leq \Per(\Omega).
\end{equation}
Moreover, for all $r>0$ the limit $\lim_{r\to 0^+}\frac{\g(0)-\g(ru) }{r}$ exists, is finite and
\begin{equation}\label{prop_v}
\Per(\Omega) =  \frac{\Gamma((d+1)/2)}{\pi^{(d-1)/2}} \int_{\mathbb{S}^{d-1}}\lim_{r\to 0^+}\frac{ \g(0)- \g(ru)}{r} \sigma (\ud u).
\end{equation}
In particular, there is a constant $C=C(\Omega)>0$ such that
\begin{align}
0\leq g_\Omega (0)- g_\Omega (y) \leq C (1\wedge |y|).\label{g_Omega_bound}
\end{align}

By Cygan, Grzywny \cite[Lemma 3]{MR3563041}, the related function $H(t)$ has the following form	
\begin{equation}\label{H_formula}
H(t) = \int _{\Rd}\left( g_\Omega (0) -\g (y)\right) p_t(\ud y),
\end{equation}
and by \cite[Proof of Lemma 1]{MR3563041},
\begin{equation}\label{perX}
\Per_{\nu} (\Omega) = \int_{\Rd} \left(\g(0)-\g(y)\right) \nu(\ud y).
\end{equation}
By \cite[Theorem 3]{MR3563041}, if $\Omega \subset \Rd$ is an open set such that $|\Omega| < \infty$ and $\Per(\Omega)<\infty$ (i.e. $\mathds{1}_{\Omega} \in \operatorname{BV}(\Rd)$),  then 
\begin{align}\label{eq222}
t^{-1}H(t) = t^{-1}(\g(0) - P_t\g (0)),
\end{align}
which converges to $ -\calL \g(0) = \Per _{\nu}(\Omega)$ as $t$ tends to $0$.

\section{Main results and proofs}
\subsection{Convolution semigroups for nonlocal operators on $\Rd$}
\begin{lemma}\label{thm:lem2}
Assume that $\psi^* \in \WUSC{\al}{1}{\uC}$ for some $\al \in (0,1)$.
If $f \in \calC_0^{\beta}$ for some $\beta \in (\al,1]$, then $\calL f \in \calC_0^{\beta-\al}$ and $\|\calL f\|_{\beta-\al} \leq C_1 (1-\al/\beta)^{-1}  h(1) \|f\|_{\beta}$, where $C_1=C_1(c_d, \overline{C})$.
In particular, if $\beta \in (2\al, 1]$, then $\calL f \in \mathcal{D}(\calL)$. 
\end{lemma}

\begin{proof}
By Lemma \ref{lem:equiv}, for all $\lambda \leq 1$ and $r \leq 1$,
\begin{equation}\label{eq:h_scaling}
h(\lambda r)\leq c \lambda^{-\al}h(r),
\end{equation}
where $c= c_d\uC$, and for all $\eps>0$, 
\begin{equation}
\int_{|y|<1} |y|^{\al+\eps} \nu(\ud y) < \infty.
\end{equation}
First, we will deal with $\vertiii{\calL f}_{\beta -\al }$. Consider $|x-y|\leq 1$. By \eqref{eq:LH},
\begin{align}
|\calL f(x)-\calL f(y)| &= \left|\int_{\Rd} \left(f(x+z)-f(x)-(f(y+z)-f(y))\right) \nu(\ud z) \right| \nonumber\\
& \leq \int_{\Rd} \left|f(x+z)-f(x)-f(y+z)+f(y)\right|  \nu(\ud z). \label{eq:int}
\end{align}
We split the integral above as follows
$$
\int_{\Rd} = \int_{|z| \leq |x-y|} + \int_{|z| > |x-y|} =: \mathrm{I}_1 + \mathrm{I}_2.
$$
We will first deal with $\mathrm{I}_1$. Denote $L= \vertiii{f}_{\beta}$. We have
\begin{equation}\label{eq:Holder}
|f(x)-f(y)|\leq L |x-y|^{\beta}.
\end{equation}
By 
\eqref{eq:Holder}, Fubini theorem, \eqref{eq:h_estimate} and  \eqref{eq:h_scaling} we have
\begin{align*}
\mathrm{I}_1 
&\leq  \int_{|z| \leq |x-y|} \left(|f(x+z)-f(x)|+|f(y+z)-f(y)|\right) \nu(\ud z) \leq 2 L \int_{|z| \leq |x-y|} |z|^{\beta} \nu(\ud z) \\
 &= 2L \int_{|z| \leq |x-y|} \int_0^{|z|^{\beta}} \ud s \nu(\ud z) = 2L\int_0^{|x-y|^{\beta}}  \int_{s^{1/\beta} \leq |z| \leq |x-y|} \nu(\ud z) \ud s \leq 2L\int_0^{|x-y|^{\beta}}  \int_{|z| \geq s^{1/\beta}} \nu(\ud z) \ud s \\
&\leq 2L\int_0^{|x-y|^{\beta}} h(s^{1/\beta}) \ud s \leq 2Lch(1) \int_0^{|x-y|^{\beta}} s^{-\al/\beta} \ud s = 2Lch(1) (1-\al/\beta)^{-1} |x-y|^{\beta-\al}.
\end{align*}
Now we will estimate $\mathrm{I}_2$. Using again 
\eqref{eq:Holder}, \eqref{eq:h_estimate} and  \eqref{eq:h_scaling} we get
\begin{align*}
\mathrm{I}_2 
&\leq \int_{|z| > |x-y|}  \left(|f(x+z)-f(y+z)| + |f(x)-f(y)|\right) \nu(\ud z) \\
&\leq 2L |x-y|^{\beta} \int_{|z| > |x-y|} \nu(\ud z) \leq 2L |x-y|^{\beta} h(|x-y|) \leq 2L c h(1) |x-y|^{\beta-\al}.
\end{align*} 
Hence
\begin{equation}\label{eq:fest1}
\vertiii{\calL f}_{\beta -\al } \leq \left(1+ (1-\al/\beta)^{-1}\right) 2ch(1) \vertiii{f}_{\beta}.
\end{equation}

Futhermore,
\begin{align*}
|\calL f(x)| \leq \int_{\Rd} |f(x+y) -f(x)| \, \nu (\ud y).
\end{align*}
We split the integral above as follows
$$
\int_{\Rd} = \int_{|y|<1} + \int_{|y|\geq1} =: \mathrm{I}_3 + \mathrm{I}_4.
$$
Proceeding as in the case of $\mathrm{I}_1$, we obtain
\begin{align*}
\mathrm{I}_3 \leq \vertiii{f}_{\beta} \int_{|y|<1} |y|^{\beta} \nu (\ud y) \leq \vertiii{f}_{\beta} \frac{c}{1-\al/\beta} h(1).
\end{align*}
Next,
\begin{align*}
\mathrm{I}_4 \leq  2 \|f\|_{\infty}  \int_{|y|\geq1} \nu (\ud y) \leq 2\|f\|_{\infty} h(1).
\end{align*}
Therefore
\begin{align}
\|\calL f\|_{\infty} \leq \frac{c}{1-\al/\beta} h(1) \vertiii{f}_{\beta} + 2h(1) \|f\|_{\infty} &\leq \left(\frac{c}{1-\al/\beta} +2\right) h(1) \|f\|_{\beta} \nonumber\\ 
&\leq \frac{3c}{1-\al/\beta} h(1) \|f\|_{\beta} \label{eq:lfe1}.
\end{align}
By \eqref{eq:fest1} and \eqref{eq:lfe1},
\begin{align}
\|\calL f\|_{\beta-\al} &\leq 2h(1) \|f\|_{\infty} + \left(2+ \frac{3}{1-\al/\beta}\right) ch(1) \vertiii{f}_{\beta} \nonumber\\
&\leq \left(2+ \left(2+\frac{3}{1-\al/\beta}\right) c\right) h(1) \|f\|_{\beta} \leq \frac{7c}{1-\al/\beta} h(1) \|f\|_{\beta} \label{eq:lfe2}.
\end{align}
The proof is complete.
\end{proof}

\begin{corollary}\label{thm:cor1}
Assume that $\psi^* \in \WUSC{\al}{1}{\uC}$ for some $\al \in (0,1)$.
If $f \in \calC_0^{\beta}$ for some $\beta \in (n\al,1]$, then $\calL^k f \in \calC_0^{\beta-k\al}$ for $k \in \{1, \ldots, n\}$ and $\calL^k f \in \mathcal{D}(\calL)$ for $k \in \{1, \dots, n-1\}$.
\end{corollary}

It is well-known that for $f \in \mathcal{D}(\calL)$ and $t\geq 0$, $P_{t}$ is differentiable and $\frac{\ud}{\ud t} P_tf = \calL P_t f = P_t \calL f$, see e.g. Pazy \cite[Theorem 1.2.4 c)]{MR710486}. Therefore, if $\calL^k f \in \mathcal{D}(\calL)$ for $k \in \{1, \dots, n-1\}$, then $\frac{\ud^n}{\ud t^n} P_tf = \calL^{n} P_t f = P_t \calL^{n} f$.
To apply this result, we will use the fact that for $t_0>0$, $P_{t_0}(\calC^{\beta}_0) \subset \calC^{\beta}_0$. Indeed, 
\begin{align*}
|P_{t_0}f(x)-P_{t_0}f(y)| &\leq \int_{\Rd} \left|f(x+z)-f(x)-f(y+z)+f(y)\right| p_{t_0} (\ud z) \\
&\leq \int_{\Rd} \left(|f(x+z)-f(y+z)| + |f(y)-f(x)| \right) p_{t_0} (\ud z) \\
&\leq 2L |x-y|^{\beta} \int_{\Rd} p_{t_0}(\ud z) = 2L |x-y|^{\beta}.
\end{align*}

\begin{theorem}\label{thm:thm1}
Assume that $\psi^* \in \WUSC{\al}{1}{\uC}$ for some $\al \in (0,1)$.
If $f \in \calC_0^{\beta}$ for some $\beta \in (n\al,1]$, then 
$$
\lim_{t \to 0^+} t^{-n} \, \left(P_tf - \sum_{k=0}^{n-1} \frac{t^k}{k!} \calL^k f \right) = \frac{1}{n!} \calL^n f.
$$
\end{theorem}
\begin{proof}
By Corollary \ref{thm:cor1}, $\calL^k f \in \mathcal{D}(\calL)$ for $k \in \{1, \dots, n-1\}$, hence $P_t f$ is $n$ times differentiable. By Taylor's theorem applied to $t \mapsto P_tf$,
$$
P_t f = \sum_{k=0}^{n-1} \frac{t^k}{k!} \calL^k f + \frac{t^n}{n!} P_{\theta_0} \calL^n f
$$
for some $\theta_0 \in (0,t)$. The thesis follows from the right continuity of $P_t$ at $t=0$.
\end{proof}
Theorem \ref{thm:thm1}, Lemma \ref{lem:equiv} and \eqref{prop_iv} give the following result.
\begin{corollary}\label{thm:cor2}
Assume that
there exists $C>0$ such that for all $\lambda\leq 1$ and $r<1$, $h(\lambda r) \leq C \lambda^{-\alpha} h(r)$
for some $\al \in (0,1)$.
Let $n\geq 2$. If $n\al<1$, then 
\begin{align*}
\lim_{t \to 0^+} t^{-n} \left(H(t) - t\Per_{\nu}(\Omega) + \sum_{k=2}^{n-1}\frac{t^k}{k!} \calL^k \g(0)\right) = -\frac{1}{n!} \calL^n \g(0).
\end{align*}
\end{corollary}

\begin{example}\label{ex_stable}
If $\nu^{(\al)}$ is an $\al$-stable L\'evy measure, $\al \in (0,1)$, then the H\"older space $\calC_0^{\beta}$ is contained in the domain of $\calL =-(-\Delta)^{\al/2}$ for any $\beta \in (\al,1]$, and we have
\begin{equation*}
\calL f(x) = \int_{\Rd\setminus\{0\}}  (f(x+y)-f(x)) \, \nu^{(\al)}(\ud y), \qquad f \in \calC_0^{\beta},\; x \in \Rd.
\end{equation*}
The associated semigroup $(p_t)_{t\geq0}$ is the $\al$-stable semigroup in $\Rd$, determined by $\psi(\xi) = |\xi|^{\al}$. We have $\psi(r \xi) = r^{\al}\psi(\xi)$, so in particular $\psi^* \in \WUSC{\al}{0}{1}$. If $f \in \calC_0^{\beta}$ for some $\beta \in (n\al,1]$, then by Corollary \ref{thm:cor1}, $\calL^k f \in \mathcal{D}(\calL)$ for $k \in \{1, \dots, n-1\}$ and $\calL^k f \in \calC_0^{\beta-k\al}$ for $k \in \{1, \ldots, n\}$, and by Theorem \ref{thm:thm1}, 
\begin{equation}\label{eq:stable}
\lim_{t \to 0^+} t^{-n} \, \left(-P_tf + \sum_{k=0}^{n-1} \frac{(-1)^{k-1} t^k}{k!} \left(-(-\Delta)^{k\al/2}\right) f \right) = \frac{(-1)^{n}}{n!} \left(-(-\Delta)^{n\al/2}\right)f,
\end{equation}
since $\left((-\Delta)^{\al/2}\right)^n =(-\Delta)^{n\al/2}$ for $n\al<2$, see \cite[(1.1.12)]{MR0350027}.
By Corollary \ref{thm:cor2},
\begin{align*}
\lim_{t \to 0^+} t^{-n} \left(H(t) - \sum_{k=1}^{n-1}\frac{(-1)^{k-1} t^k}{k!} \Per_{\nu^{(k\al)}} (\Omega)\right) = \frac{(-1)^{n-1}}{n!} \Per_{\nu^{(n\al)}} (\Omega).
\end{align*}
\end{example}

\begin{theorem}\label{thm:thm2}
Assume that for all $\al \in (0,1)$, $\psi^* \in \WUSC{\al}{1}{C \al^{-1}}$ for some $C>0$.
If $f \in \calC_0^{\beta}$ for some $\beta \in (0,1]$, then there exists $t_0>0$ such that for all $t\in(0,t_0)$,
$$
P_t f = \sum_{k=0}^{\infty} \frac{t^k}{k!} \calL^k f
$$
in $\calC_0(\Rd)$.
\end{theorem}
\begin{proof}
Without loss of generality, we can assume $h(1)=1$. 
For any $N \in \mathbb{N}$, let $\alpha = \alpha(N) = \beta/(2N)$. For $n \in \{1, 2, \ldots, N-1\}$, let $\beta_n = \beta - n\alpha$. 
Since $N\alpha = \beta/2 < \beta$, by the Proof of Theorem \ref{thm:thm1},
\begin{align*}
P_t f (x) = \sum_{n=0}^{N-1} \frac{t^n \calL^n f (x)}{n!} + \tilde{R}_N,
\end{align*}
where 
$$
\tilde{R}_N = \frac{t^N P_{\theta_0} \calL^N f(x)}{N!}.
$$
By \eqref{eq:lfe1}, for any $N \in \mathbb{N}$,
\begin{equation}\label{eq:lnf}
\|\calL^N f\|_{\infty}  
\leq \frac{7 C/\alpha}{1-\al/\beta_{N-1}}  \|\calL^{N-1} f\|_{\beta_{N-1}} .
\end{equation}
By \eqref{eq:lfe2}, for $k \in \{1, 2, \ldots N-1\}$,
\begin{equation}\label{eq:lfb}
\|\calL f\|_{\beta_k} 
\leq \frac{7C/\alpha}{1-\al/\beta_{k-1}} \|f\|_{\beta_{k-1}}.
\end{equation}
Using \eqref{eq:lnf} and applying \eqref{eq:lfb} $N$ times we get
\begin{align}
\|\calL^N f \|_{\infty} \leq \frac{7C/\al}{1-\al/\beta_{N-1}}
\|\calL^{N-1} f\|_{\beta_{N-1}} 
&\leq \left(\prod_{k=0}^{N-1}  \frac{7C/\al}{1-\al/\beta_k}
\right)  \|f\|_{\beta} \nonumber\\
&\leq  (7C)^N \left(\prod_{k=0}^{N-1} \frac{1/\al}{1-2\al/\beta}\right)  \|f\|_{\beta} \nonumber\\
&
= (7C)^N \left(\prod_{k=0}^{N-1} \frac{2N/\beta}{1-1/N}\right) 
 \|f\|_{\beta} \nonumber \\
&
=(14C/\beta)^N N^N\left(\prod_{k=0}^{N-1} \frac{1}{1-1/N}\right) 
\|f\|_{\beta} \nonumber\\
&\leq (14C/\beta)^N N^N e \|f\|_{\beta}\label{eq:lnfe}.
\end{align}
By the contractivity of $P_{\theta_0}$, \eqref{eq:lnfe} and Stirling's formula,
\begin{align*}
|\tilde{R}_N|& \leq \frac{\|t^N P_{\theta_0} \calL^N f \|_{\infty}}{N!} \leq \frac{t^N\|\calL^N f \|_{\infty}}{N!} 
\leq \frac{c C'^N N^N t^N}{\sqrt{2\pi N} (N/e)^N} \|f\|_{\beta} \leq \frac{c'\tilde{C}^N t^N}{\sqrt{N}} \|f\|_{\beta} 
\end{align*}
which tends to $0$ as $N \to \infty$. The proof is complete.
\end{proof}
\begin{example}
Let $\psi(\xi)=\log(1+|\xi|^2)$, i.e. $\calL =- \log(1 - \Delta)$.
Let $\alpha\in(0,1]$. Since, for $\lambda\geq 1$ and $x\geq 1$,
\begin{align*}
\log(1+\lambda x) &\leq \log(\lambda(1+x))
= \frac{1}{\alpha} \log(\lambda^{\alpha}(1+x)^{\alpha}) 
\leq \frac{1}{\alpha} \log(\lambda^{\alpha} (1+x)),
\end{align*}
and 
$$
\frac{\log(\lambda^{\alpha} (1+x))}{ \log(1+x)}\leq \frac{ \lambda^{\alpha}}{ \log 2},
$$
we have $\log(1+\cdot) \in \WUSC{\al}{1}{2/\al}$. Hence $\psi \in \WUSC{\al}{1}{4/\al}$,
that is $\psi$ satisfies the assumptions of Theorem \ref{thm:thm2}.
\end{example}

\begin{corollary}\label{cor3}
Assume that for all $\al \in (0,1)$ and $\lambda \leq 1$, $h(\lambda r) \leq C \al^{-1} \lambda^{-1} h(r)$ for some $C>0$.
Then, there exists $t_0>0$ such that for all $t\in(0,t_0)$,
\begin{align*}
H(t) = t\Per_{\nu}(\Omega) - \sum_{k=2}^{\infty}\frac{t^k}{k!} \calL^k \g(0).
\end{align*}
\end{corollary}

\begin{example}
Let $\nu$ be a finite measure on $\Rd$ and let $(p_t)_{t\geq0}$ be determined by
$$
\psi(\xi) = \int_{\Rd} (1-e^{i \langle \xi, z \rangle}) \, \nu(\ud z).
$$
In this case
\begin{align*}
\calL f (x) 
&= \int_{\Rd} \left(f(x+y)-f(x)\right) \nu(\ud y).
\end{align*}
Generator $\calL$ can be expressed as a convolution operator
$$
\calL f = \left(\nu - \nu(\Rd) \delta_0 \right) * f,
$$
therefore
$$
\calL^n f = (\nu - \nu(\Rd) \delta_0)^{*n} * f = \sum_{i=0}^{n} (-1)^{n-i} \binom{n}{i} \nu(\Rd)^{n-i} \nu^{*i} * f,
$$
where for $k\in\mathbb{N}$, $\mu^{*k}$ denotes the $k$-fold iteration of the convolution of measure $\mu$ with itself, i.e. $\mu^{*0} = \delta_0$ and $\mu^{*k} = \mu^{*(k-1)} * \mu$ for $k\geq1$.
It is well-known that, since $\nu$ is finite and $\calL$ is bounded, we have
$$
P_t = e^{t\calL}.
$$
Therefore
\begin{align*}
P_t = \sum_{n=0}^{\infty} \frac{t^n}{n!} \calL^n &= \sum_{n=0}^{\infty} \sum_{i=0}^{n} (-1)^{n-i} \frac{t^n}{i!(n-i)!}\nu(\Rd)^{n-i} \nu^{*i} 
\\
&=\sum_{i=0}^{\infty} \sum_{n=i}^{\infty} (-1)^{n-i} \frac{t^n}{i!(n-i)!}\nu(\Rd)^{n-i} \nu^{*i} \\
&= \sum_{j=0}^{\infty} \frac{(-t\nu(\Rd))^{j}}{j!} \sum_{i=0}^{\infty} \frac{t^i}{i!} \nu^{*i}  = 
 e^{-t\nu(\Rd)} \exp(t\nu),
\end{align*}
where $\exp(\nu) = \sum_{n=0}^{\infty} \frac{1}{n!} \nu^{*n}$. This expansion follows also from Theorem \ref{thm:thm2}. Applying this result to $f=\g$, we extend \cite[Theorem 1.2]{MR3641640}, which holds for compactly supported probabilistic measures with radial density, to general finite measures.
\end{example}

\subsection{Heat content for the fractional Laplacian on $\Rd$}

Let $(p_t)_{t\geq 0}$ be the $\al$-stable semigroup in $\Rd$, $\al \in (0,2)$.  We recall that in this case $\psi(\xi) = |\xi|^{\al}$ and the corresponding L\'{e}vy measure $\nu$ is the $\al$-stable L\'{e}vy measure $\nu^{(\al)}$.
The related function $h$ turns into $h(r) = c/r^{\al}$, for some $c>0$.

Let 
$$
a_n := \frac{1}{\pi^{1+d/2}}\frac{(-1)^{n-1}}{n!} 2^{n\al} \Gamma\left(\frac{n\al}{2}+1\right) \Gamma\Big(\frac{n\al+d}{2}\Big)  \sin\left(\frac{\pi n \al}{2}\right). 
$$
For $\frac{n\al}{2} \notin \mathbb{N}$,
$$
a_n = \frac{(-1)^{n-1}}{n!} \mathcal{A}_{d,-n\al}.
$$
By Hiraba \cite[Remark 2.b)]{MR1287843}, for $\al<1$ and $x \in \Rd\setminus\{0\}$,
\begin{equation}\label{eq:Hiraba}
p_1(x) = \sum_{n=1}^{\infty} a_n |x|^{-n\al-d}. 
\end{equation}

Our following two results extend \cite[Theorem 1.2]{MR3606559}. Note that they provide more detailed expansion than the one resulting from Corollary \ref{thm:cor2}, compare with Example \ref{ex_stable}.

\begin{theorem}\label{thm3}
Let $\al \in (0,1)$ be such that $1/\al \notin \mathbb{N}$ and let $\Omega \subset \Rd$ be an open set of finite Lebesgue measure and perimeter. Then 
\begin{align*}
\lim_{t \to 0^+} t^{-1/\al}& \left( H(t) - \sum_{n=1}^{\left[\frac{1}{\al}\right]} \frac{(-1)^{n-1}}{n!} t^n \Per_{\nu^{(n\al)}} (\Omega) \right)\\
& \qquad\qquad\qquad=\frac{\pi^{\frac{d-1}{2}}}{\Gamma\left(\frac{d+1}{2}\right)} \Per(\Omega)  \left(\int_0^1 r^d p_1(re_d) \, \ud r -\sum_{n=1}^{\infty} \frac{a_n}{1-n\al} \right).
\end{align*}
\end{theorem}

\begin{proof}
Without loss of generality, we can assume $\mathrm{diam}(\Omega)=1$. For $1/\al \notin \N$, $\left[1/\al\right] = \lceil1/\al\rceil -1$, and we will use this formula in order to avoid repeating similar calculations in the next proof. By \eqref{H_formula}, \eqref{perX} and the scaling property of $p_t$,
\begin{align*}
H(t) - \sum_{n=1}^{\lceil\frac{1}{\al}\rceil -1} &\frac{(-1)^{n-1}}{n!} t^n \Per_{\nu^{(n\al)}} (\Omega) \\
&= \int_{\Rd} \left(\g(0)-\g(x)\right) \left(p_t(x) - \sum_{n=1}^{\lceil\frac{1}{\al}\rceil -1} \frac{(-1)^{n-1}}{n!} t^n \mathcal{A}_{d,-n\al} |x|^{-d-n\al} \right) \ud x \\
&=\int_{\Rd} \left(\g(0)-\g(t^{1/\al}x)\right) \left(p_1(x) - \sum_{n=1}^{\lceil\frac{1}{\al}\rceil -1} \frac{(-1)^{n-1}}{n!} \mathcal{A}_{d,-n\al} |x|^{-d-n\al} \right)  \ud x.
\end{align*}
We split the above integral into
\begin{equation}\label{int_split}
\int_{|x| \leq 1} +\int_{1 <|x| \leq t^{-1/\al}} + \int_{|x| > t^{-1/\al}} =: \mathrm{I}_1+\mathrm{I}_2+\mathrm{I}_3.
\end{equation}
We have
\begin{align*}
\mathrm{I}_1 &= \int_{|x| \leq 1} \left(\g(0)-\g(t^{1/\al}x)\right) \Bigg(p_1(x) - \sum_{n=1}^{\lceil\frac{1}{\al}\rceil -1} \frac{(-1)^{n-1}}{n!} \mathcal{A}_{d,-n\al} |x|^{-d-n\al} \Bigg)  \ud x \\
&= t^{1/\al}\int_{0}^{1} \int_{\mathbb{S}^{d-1}} r^d \, \frac{\g(0)-\g(t^{1/\al}r u)}{t^{1/\al}r} \Bigg(p_1(re_d) - \sum_{n=1}^{\lceil\frac{1}{\al}\rceil -1} \frac{(-1)^{n-1}}{n!} \mathcal{A}_{d,-n\al} r^{-d-n\al} \Bigg)  \sigma(\ud u) \, \ud r.
\end{align*}
By \eqref{prop_iv}, \eqref{prop_v} and Dominated Convergence Theorem,			
$$
\lim_{t \to 0^+} t^{-1/\al} \mathrm{I}_1 = \frac{\pi^{\frac{d-1}{2}}}{\Gamma\left(\frac{d+1}{2}\right)} \Per(\Omega) \left( \int_0^1 r^d p_1(re_d) \ud r - \sum_{k=1}^{\lceil\frac{1}{\al}\rceil -1} \frac{(-1)^{n-1}}{n!} \frac{\mathcal{A}_{d,-n\al}}{1-n\al}  \right) .
$$
Next,
\begin{align*}
\frac{\mathrm{I}_3}{\g(0)} &=   \int_{|x| > t^{-1/\al}} \left(p_1(x) - \sum_{n=1}^{\lceil\frac{1}{\al}\rceil -1} \frac{(-1)^{n-1}}{n!} \mathcal{A}_{d,-n\al} |x|^{-d-n\al} \right) \ud x \\
& = \int_{|x| > t^{-1/\al}} \sum_{n=\lceil\frac{1}{\al}\rceil}^{\infty} a_n |x|^{-n\al-d} \, \ud x\\
&= \sum_{n=\lceil\frac{1}{\al}\rceil}^{\infty} a_n \int_{|x| > t^{-1/\al}}|x|^{-n\al-d} \, \ud x 
=  \omega_{d-1} \sum_{n=\lceil\frac{1}{\al}\rceil}^{\infty} \frac{a_n}{n\al}  t^{n}.
\end{align*}
We have
$$
|\mathrm{I}_3| \leq \g(0)\omega_{d-1} \sum_{n=\lceil\frac{1}{\al}\rceil}^{\infty} \frac{|a_n|}{n\al}  t^{n}=  O(t^{\lceil\frac{1}{\al}\rceil})
$$
for $t<1$,
thus
$$
\lim_{t\to 0^+} t^{-1/\al} \mathrm{I}_3 = 0.
$$
We have
\begin{align*}
\mathrm{I}_2 &= \int_{1< |x| < t^{-1/\alpha}} \left(\g(0)-\g(t^{1/\al}x)\right) \Bigg(p_1(x) - \sum_{n=1}^{\lceil\frac{1}{\al}\rceil -1} \frac{(-1)^{n-1}}{n!} \mathcal{A}_{d,-n\al} |x|^{-d-n\al} \Bigg)  \ud x \\
&= t^{1/\al}\int_{1}^{t^{-1/\alpha}} \int_{\mathbb{S}^{d-1}} r^d \, \frac{\g(0)-\g(t^{1/\al}r u)}{t^{1/\al}r} \Bigg(p_1(re_d) - \sum_{n=1}^{\lceil\frac{1}{\al}\rceil -1} \frac{(-1)^{n-1}}{n!} \mathcal{A}_{d,-n\al} r^{-d-n\al} \Bigg)  \sigma(\ud u) \, \ud r \\
&= t^{1/\al}\int_{1}^{t^{-1/\alpha}} \int_{\mathbb{S}^{d-1}} \, \frac{\g(0)-\g(t^{1/\al}r u)}{t^{1/\al}r} \sum_{n=\lceil\frac{1}{\al}\rceil}^{\infty} a_n \, \sigma(\ud u) \, r^{-n\al}\, \ud r.
\end{align*}
By \eqref{prop_iv}, 
\begin{align*}
|\mathrm{I}_2| &\leq  \frac{\Per(\Omega)}{2} t^{1/\al} \sum_{n=\lceil\frac{1}{\al}\rceil}^{\infty} |a_n| \int_{1<|x|\leq t^{-1/\al}} |x|^{1-n\al-d} \, \ud x \\
&\leq \frac{\Per(\Omega)}{2}  t^{1/\al} \sum_{n=\lceil\frac{1}{\al}\rceil}^{\infty} |a_n| \int_{|x|>1} |x|^{1-n\al-d} \, \ud x \\
&=\frac{\Per(\Omega)}{2}  \omega_{d-1} t^{1/\al} \sum_{n=\lceil\frac{1}{\al}\rceil}^{\infty} \frac{|a_n|}{n\al-1},
\end{align*}
hence
$$
\mathrm{I}_2 = \sum_{n=\lceil\frac{1}{\al}\rceil}^{\infty} a_n \int_{1 <|x| \leq t^{-1/\al}} \left(\g(0)-\g(t^{1/\al}x)\right)  |x|^{-n\al-d} \, \ud x
$$
and
$$
\lim_{t \to 0^+} t^{-1/\al} \mathrm{I}_2 = \sum_{n=\lceil\frac{1}{\al}\rceil}^{\infty} a_n \lim_{t \to 0^+} t^{-1/\al} \int_{1 <|x| \leq t^{-1/\al}} \left(\g(0)-\g(t^{1/\al}x)\right)  |x|^{-n\al-d} \, \ud x.
$$
We get
\begin{align*}
t^{-1/\al} \mathrm{I}_2= \sum_{n=\lceil\frac{1}{\al}\rceil}^{\infty} a_n \int_1^{t^{-1/\al} }  \calM_{\Omega}(t,r) r^{-n\al}\, \ud r, 
\end{align*}
where 
\begin{align}\label{M_om}
\calM_{\Omega}(t,r) = \int_{\mathbb{S}^{d-1}} \frac{g_\Omega (0) - g_\Omega \left(rt^{1/\al}u\right)}{rt^{1/\al}}\,
\sigma (\ud u).
\end{align}
We claim that 
\begin{align}\label{claim-new}
\lim_{t \to 0^+} \int_1^{t^{-1/\al}} \calM_{\Omega}(t,r) r^{-n\al}\, \ud r =  \frac{\pi^{(d-1)/2}}{\Gamma\left((d+1)/2\right)}\Per (\Omega) \frac{1}{n\al-1}.
\end{align}
Indeed, by \eqref{prop_iv} and \eqref{prop_v},
\begin{align}\label{M_om1}
0\leq \calM_{\Omega}(t,r)\leq \frac{1}{2} \Per(\Omega)\, \sigma (\mathbb{S}^{d-1}) 
\end{align}
and, for any $r>0$,
\begin{align}\label{M_om2}
\lim_{t\to 0^+}\calM_{\Omega}(t,r) = \frac{\pi^{(d-1)/2}}{\Gamma\left((d+1)/2\right)}\Per (\Omega).
\end{align}
Moreover,
$$
\int_1^{t^{-1/\al}} r^{-n\al}\, \ud r \leq \int_1^{\infty} r^{-n\al}\, \ud r = \frac{1}{n\al-1}
$$
and hence \eqref{claim-new} follows by Dominated Convergence Theorem.
\end{proof}

\begin{theorem}\label{thm4}
Let $\al \in (0, 1)$ be such that $1/\al \in \mathbb{N}$ and let $\Omega \subset \Rd$ be an open set of finite Lebesgue measure and perimeter.
Then 
\begin{align*}
\lim_{t \to 0^+} (t^{1/\al}\log(1/t))^{-1}  \left( H(t) - \sum_{n=1}^{1/\al-1} \frac{(-1)^{n-1}}{n!} t^n \Per_{\nu^{(n\al)}} (\Omega) \right)
& = \frac{(-1)^{1/\alpha-1}}{(1/\alpha-1)!\pi}\Per(\Omega).
\end{align*}
\end{theorem}

\begin{proof}
By the Proof of Theorem \ref{thm3},
\begin{align*}
H(t) - \sum_{n=1}^{1/\al-1} &\frac{(-1)^{n-1}}{n!} t^n \Per_{\nu^{(n\al)}} (\Omega) = \mathrm{I}_1+\mathrm{I}_2+\mathrm{I}_3,
\end{align*}
where
\begin{align*}
\mathrm{I}_1 = t^{1/\al}\int_{0}^{1} \int_{\mathbb{S}^{d-1}} r^d \, \frac{\g(0)-\g(t^{1/\al}r u)}{t^{1/\al}r} \left(p_1(re_d) - \sum_{n=1}^{1/\al-1} \frac{(-1)^{n-1}}{n!} \mathcal{A}_{d,-n\al} r^{-d-n\al} \right)  \sigma(\ud u) \, \ud r,
\end{align*}
\begin{align*}
\mathrm{I}_2&= \int_{1 <|x| \leq t^{-1/\al}} \left(\g(0)-\g(t^{1/\al}x)\right) \sum_{n=1/\al}^{\infty} a_n |x|^{-n\al-d} \, \ud x,
\end{align*}
and
\begin{align*}
\mathrm{I}_3  = |g_{\Omega}(0)| \omega_{d-1} \sum_{n=1/\al}^{\infty} \frac{a_n}{n\al} t^{n}.
\end{align*}
By \eqref{prop_iv}, \eqref{prop_v} and  Dominated Convergence Theorem,			
$$
\lim_{t \to 0^+} \left(t^{1/\al} \log(1/t)\right)^{-1} \mathrm{I}_1 = 0.
$$
Next,
$$
\lim_{t\to 0^+} (t^{1/\al} \log(1/t))^{-1} \mathrm{I}_3 = 0.
$$

By \eqref{prop_iv}, 
\begin{align*}
|\mathrm{I}_2| &\leq  \frac{\Per(\Omega)}{2} t^{1/\al} \sum_{n=1/\al}^{\infty} |a_n| \int_{1<|x|\leq t^{-1/\al}} |x|^{1-n\al-d} \, \ud x \\
&= \frac{\Per(\Omega)}{2}   \left(\sum_{n=1/\al+1}^{\infty} |a_n| t^{1/\al} \frac{t^{n-1/\al}-1}{1-n\al} + \frac{a_{1/\al}}{\al} t^{1/\al} \log(1/t)\right) \\
&\leq \frac{\Per(\Omega)}{2}   \left(\sum_{n=1/\al+1}^{\infty} |a_n|  \frac{t^{1/\al}}{n\al-1} +  \frac{a_{1/\al}}{\alpha} t^{1/\al} \log(1/t)\right).
\end{align*}
Therefore
$$
\mathrm{I}_2 = \sum_{n=1/\al}^{\infty} a_n \int_{1 <|x| \leq t^{-1/\al}} \left(\g(0)-\g(t^{1/\al}x)\right)  |x|^{-n\al-d} \, \ud x.
$$
We have
\begin{align*}
&(t^{1/\al} \log(1/t))^{-1} \int_{1 <|x| \leq t^{-1/\al}} \left(\g(0)-\g(t^{1/\al}x)\right)  |x|^{-n\al-d} \, \ud x \\
&= \log(1/t)^{-1}\int_1^{t^{-1/\al}} \int_{\mathbb{S}^{d-1}} \frac{\g(0)-\g(t^{1/\al}ru)}{t^{1/\al} r} \sigma(\ud u) r^{-n\al} \ud r \\
&= \log(1/t)^{-1}\int_1^{t^{-1/\al}} \calM_{\Omega}(t,r) r^{-n\al} \ud r.
\end{align*}
We claim that 
\begin{align}\label{eq:lim_log1}
\lim_{t\to 0^+} \, \log(1/t)^{-1} \int_1^{t^{-1/\al} }\!\!
\calM_{\Omega}(t,r) r^{-1}  \, \ud r= \frac{\pi^{\frac{d-1}{2}}}{\al \Gamma\left(\frac{d+1}{2}\right)} \Per(\Omega),
\end{align}
where 
$$
\calM_{\Omega} (t,r) = \int_{\mathbb{S}^{d-1}} \frac{\g(0)-\g(rt^{1/\al}u)}{rt^{1/\al}} \sigma (\ud u). 
$$
Indeed,
by substitution, 
\begin{align*}
\log(1/t)^{-1} \int_1^{t^{-1/\al}}  \calM_{\Omega} (t,r)r^{-1} \, \ud r &= \log(1/t)^{-1} \int_0^{ \log(1/t)/\al}  \calM_{\Omega} (t,e^r) \, \ud r 
\\
&= 
\int_0^{1/\al} \calM_{\Omega} (t,t^{-r}) \, \ud r,
\end{align*}
and by \eqref{prop_iv}, \eqref{prop_v} and Dominated Convergence Theorem,
\begin{align*}
\lim_{t \to 0^+} \int_0^{1/\al} \calM_{\Omega} (t,t^{-r}) \, \ud r
&= \int_0^{1/\al} \int_{\mathbb{S}^{d-1}} \lim_{t \to 0^+} \frac{\g(0)-\g(t^{1/\al-r}u)}{t^{1/\al-r}} \sigma (\ud u) \, \ud r  \\
&= \frac{\pi^{(d-1)/2}}{\al \Gamma((d+1)/2)} \Per(\Omega).
\end{align*}
We claim that for $n \geq 1/\al+1$ we have
\begin{align}\label{eq:lim_log2}
\lim_{t \to 0^+} \log(1/t)^{-1}\int_1^{t^{-1/\al}} \calM_{\Omega}(t,r) r^{-n\al} \ud r  = 0.
\end{align}
Indeed, by \eqref{M_om1}, \eqref{M_om1}, and since
\eqref{prop_iv} and \eqref{prop_v},
\begin{align*}
0\leq \calM_{\Omega}(t,r)\leq \frac{1}{2} \Per(\Omega)\, \sigma (\mathbb{S}^{d-1}) 
\end{align*}
and, for any $r>0$,
\begin{align*}
\lim_{t\to 0^+}\calM_{\Omega}(t,r) = \frac{\pi^{(d-1)/2}}{\Gamma\left((d+1)/2\right)}\Per (\Omega).
\end{align*}
Moreover,
$$
\int_1^{t^{-1/\al}} \log(1/t)^{-1}  r^{-n\al}\, \ud r \leq \int_1^{\infty} r^{-n\al}\, \ud r = \frac{1}{n\al-1}
$$
for $t<1/e$, and hence \eqref{eq:lim_log2} follows by Dominated Convergence Theorem. \eqref{eq:lim_log1} and \eqref{eq:lim_log2} yield the thesis of the theorem.
\end{proof}

\subsection{Heat content for general stable operators on $\R$}

Let $\al \in (0,1) \cup (1,2)$, $\beta \in [-1,1]$ and $\gamma>0$. We consider convolution semigroup $(p_t)_{t\geq0}$ on $\R$ such that 
$$
\psi(\xi) = \gamma|\xi|^{\al} \left(1-i\beta \operatorname{tg}\left(\frac{\pi\al}{2}\right) \sgn(\xi)\right), \quad \xi \in \R \,.
$$ 
The  corresponding L\'evy measure on $\R$ is given by 
$$
\nu(\ud x) = \frac{c_+ \mathds{1}_{x\geq0} + c_- \mathds{1}_{x<0}}{|x|^{1+\al}} \, \ud x,
$$
where
$$
c_+ = -\frac{1+\beta}{2\Gamma(-\al)\cos(\frac{\pi\al}{2})}\quad \mathrm{and} \quad c_- = -\frac{1-\beta}{2\Gamma(-\al)\cos(\frac{\pi\al}{2})}.
$$

Let $\Omega=(a,b) \subset \R$. We have $\g(x) = (b-a-|x|)\mathbf{1}_{[0,b-a)} (|x|)$. For $\alpha \in (0,1)$, $\Per_{\nu}(\Omega) = (c_++c_-) \al^{-1}(1-\al)^{-1} (b-a)^{1-\al}$. 

We can fix the parameter $\gamma$ without loss of generality. Assume that $\gamma= \cos\left(\frac{\pi\beta\al}{2}\right)$ if $\al<1$ and $\gamma= \cos\left(\pi\beta\frac{2-\al}{2}\right)$ if $\al>1$.

Let 
$$
b_n := \frac{(-1)^{n-1}}{\pi^{3/2}}\frac{\Gamma(n\al+1)}{n!} \sin\left(\pi n\al \rho\right),
$$
\vspace{0.3em}
where $\rho = \frac{1+\beta}{2}$ if $\al<1$, and $\rho = \frac{1-\beta(2-\al)/\al}{2}$ if $\al>1$.
By \cite[(2.5.1)]{MR854867}, for $\al < 1$,
\begin{equation}\label{eq:zolot1}
p_1(x) = \sum_{n=1}^{\infty} b_n x^{-n\al-1}
\end{equation}
as $x \to \infty$. By \cite[(2.5.4)]{MR854867}, for $\al>1$ and $\beta\neq-1$, and any $N \in \mathbb{N}$,
\begin{equation}\label{eq:zolot2}
p_1(x) = \sum_{n=1}^{N} b_n x^{-n\al-1} + O (x^{-(N+1)\al-1})
\end{equation}
as $x \to \infty$. 
Let 
$$
d_n := \frac{ (-1)^{n-1}}{\pi^{3/2}}
\frac{2\Gamma(n\al+1)}{n!} \sin\left(\frac{\pi n\al}{2}\right) \cos\left(\frac{\pi n \al \beta}{2}\right).
$$
This constant will appear in the following proposition, which complements \cite[Theorem 1.1]{MR3606559}. We generalize the results for $\al<1$ to the non-symmetric case. The last result is new even for the symmetric case, since previously only the first two terms of the expansion were known.

\begin{proposition} 
 Let $\Omega=(a,b), |\Omega| = b-a$. 
\begin{enumerate}
\item Let $0<\al<1$ and $0<t<\min \{|\Omega|^{\alpha}, e^{-1} \}$.
\begin{enumerate}
\item If $1/\al \notin \mathbb{N}$, then there is a constant $C_{\alpha}$ independent of $\Omega$ such that
\begin{align*}
H(t) &=\frac{2}{\pi} \sum_{n=1}^{\left[\frac{1}{\alpha}\right]}(-1)^{n-1} \frac{\Gamma(n\alpha)}{(1-n\alpha)n!}
\sin\left(\frac{\pi n \al}{2}\right) \cos\left(\frac{\pi n \alpha \beta}{2}\right) |\Omega|^{1-n\alpha} t^{n}+\,\,C_{\alpha}\,t^{1/\alpha} +R_{\alpha}(t),
\end{align*}  
where
$C_{\alpha} = \int_0^1 \int_{w}^{\infty} \left(p_1(x)+p_1(-x)\right) \ud x \ud w - \sum_{n=1}^{\infty} \frac{d_n}{n\al(1-n\al)}$ and $|R_{\alpha}(t)| \leq c \,t^{\left[\frac{1}{\alpha} \right]+1}$.
\item If $\alpha=1/N$ for some $N \in \mathbb{N}$, then there is a constant $C_{N}(\Omega)$ such that
\begin{align*}
H(t)&= \frac{2}{\pi} \sum_{n=1}^{N-1} (-1)^{n-1}
\frac{\Gamma(n/N)}{(1-n/N)n!}
\sin\left( \frac{\pi n}{2N}\right) \cos\left(\frac{\pi n \beta}{2N}\right)
|\Omega|^{1-n/N} t^{n}
\\&\quad+(-1)^{N-1}\,\frac{2}{\pi (N-1)!}\,
 \cos \left(\frac{\pi \beta}{2}\right)
 t^{N} \ln\left(\frac{1}{t}\right)  + C_{N}(\Omega)\,t^{N}
+R_{1/N}(t),
\end{align*}
where  $C_N(\Omega) =  \int_0^1 \int_{w}^{\infty} \left(p_1(x)+p_1(-x)\right) \ud x \ud w\, + \,  d_N \ln(|\Omega|) - \sum_{n\neq N} \frac{d_n}{n/N(1-n/N)}$, $|R_{1/N}(t)|\leq c\, t^{N+1}$.
\end{enumerate}
\item If $1<\al<2$, $|\beta| \neq 1$, 
then, for any $N \in \mathbb{N}$,
\begin{align*}
H(t) &= t^{1/\al} \int_{\R} |x|p_1(x) \ud x + \frac{2}{\pi} \sum_{n=1}^{N} (-1)^{n-1} \frac{\Gamma(n\al+1)}{n!} \sin\left(\frac{\pi n \al}{2}\right) \cos\left(\frac{\pi n 
\left(\beta\frac{2-\al}{\al}\right)}{2}\right)
\frac{1}{n\al(1-n\al)} |\Omega|^{1-n\al} t^n \\
&\quad+ R_{N}(t)
\end{align*}
as $t \to 0^+$, where $|R_N(t)| \leq c t^{N+1}$.
\end{enumerate}
\end{proposition}

Note that by \cite[Proposition 1.4]{Hardin}
$$
\int_{\R} |x|p_1(x) \ud x = \frac{2}{\pi} \Gamma\left(1-\frac{1}{\al}\right) \mathrm{Re} \left(1+i\beta\operatorname{tg}\left(\frac{\pi \al}{2} \right)\right)^{1/\al}.
$$

\begin{proof}
(i) can proved analogously to \cite[Theorem 1.1]{MR3606559}, using \eqref{eq:zolot1}. We will prove (ii). We have
We have
\begin{align*}
H(t) &= \int_{\R} \left(|\Omega| \wedge |x|\right) p_t(x) \, \ud x = \int_{\R} \left(|\Omega| \wedge t^{1/\al}|x|\right) p_1(x) \, \ud x \\
&= t^{1/\al} \int_{|x|<|\Omega|t^{-1/\al}} |x|p_1(x) \, \ud x +  \int_{|x|\geq |\Omega|t^{-1/\al}} \left(|\Omega|-t^{1/\al}|x|\right) p_1(x) \, \ud x\\
& = t^{1/\al} \, \int_{\R} |x|p_1(x) \ud x +
\int_{|x|\geq |\Omega|t^{-1/\al}} \left(|\Omega|-t^{1/\al}|x|\right) p_1(x) \, \ud x. 
\end{align*}
By \eqref{eq:zolot2}, for any $N \in \mathbb{N}$, 
$$
p_1(x) = \sum_{n=1}^{N} b_n x^{-n\al-1} + O(x^{-(N+1)\al-1}),
$$
as $x \to \infty$.
Therefore
\begin{align*}
&\int_{|\Omega|t^{-1/\al}}^{\infty} \left(|\Omega|-t^{1/\al}x\right) p_1(x) \, \ud x \\
&=  \int_{|\Omega|t^{-1/\al}}^{\infty} \left(\sum_{n=1}^{N} (-1)^{n-1} b_n x^{-n\al-1} + O (x^{-(N+1)\al-1})\right)  \left(|\Omega|-t^{1/\al}x
\right) \ud x \\
&= \sum_{n=1}^{N} \frac{b_n}{n\al(1-n\al)} |\Omega|^{1-n\al} t^n +  \int_{|\Omega|t^{-1/\al}}^{\infty} O (x^{-(N+1)\al-1})  \left(|\Omega|-t^{1/\al}x\right) \ud x.
\end{align*}
We also have
\begin{align*}
\left|\int_{|\Omega|t^{-1/\al}}^{\infty} O (x^{-(N+1)\al-1})  \left(|\Omega|-t^{1/\al}x\right) \ud x\right| &\leq C \int_{|\Omega|t^{-1/\al}}^{\infty} \left(t^{1/\al}x-|\Omega|\right)  x^{-(N+1)\al-1} \, \ud x \\
& = C \frac{1}{(N+1)\al((N+1)\al-1)} |\Omega|^{1-(N+1)\al} \, t^{N+1}.
\end{align*}
The calculations for $\int_{-\infty}^{-\Omega|t^{-1/\al}} \left(|\Omega|+t^{1/\al}x\right) p_1(x) \, \ud x$ are analogous since $p_1(-x)$ corresponds to $p_1(x)$ with parameters $(\al, -\beta, \gamma)$.
\end{proof}

\bibliographystyle{abbrv}

\begin{thebibliography}{10}

\bibitem{MR3606559}
L.~Acu{\~{n}}a~Valverde.
\newblock Heat content for stable processes in domains of {$\Bbb{R}^d$}.
\newblock {\em J. Geom. Anal.}, 27(1):492--524, 2017.

\bibitem{MR4224349}
L.~Acu{\~{n}}a~Valverde.
\newblock On the heat content for the {P}oisson kernel over the unit ball in
  the euclidean space.
\newblock {\em Bull. Lond. Math. Soc.}, 52(6):1093--1104, 2020.

\bibitem{MR4158754}
L.~Acu{\~{n}}a~Valverde.
\newblock On the heat content for the {P}oisson heat kernel over convex bodies.
\newblock {\em J. Math. Anal. Appl.}, 494(2):Paper No. 124655, 15, 2021.

\bibitem{MR1857292}
L.~Ambrosio, N.~Fusco, and D.~Pallara.
\newblock {\em Functions of bounded variation and free discontinuity problems}.
\newblock Oxford Mathematical Monographs. The Clarendon Press, Oxford
  University Press, New York, 2000.

\bibitem{MR3019137}
L.~Angiuli, U.~Massari, and M.~Miranda, Jr.
\newblock Geometric properties of the heat content.
\newblock {\em Manuscripta Math.}, 140(3-4):497--529, 2013.

\bibitem{MR2675483}
L.~Caffarelli, J.-M. Roquejoffre, and O.~Savin.
\newblock Nonlocal minimal surfaces.
\newblock {\em Comm. Pure Appl. Math.}, 63(9):1111--1144, 2010.

\bibitem{MR3563041}
W.~Cygan and T.~Grzywny.
\newblock Heat content for convolution semigroups.
\newblock {\em J. Math. Anal. Appl.}, 446(2):1393--1414, 2017.

\bibitem{MR3859849}
W.~Cygan and T.~Grzywny.
\newblock A note on the generalized heat content for {L}\'{e}vy processes.
\newblock {\em Bull. Korean Math. Soc.}, 55(5):1463--1481, 2018.

\bibitem{MR3732178}
F.~Ferrari, M.~Miranda, Jr., D.~Pallara, A.~Pinamonti, and Y.~Sire.
\newblock Fractional {L}aplacians, perimeters and heat semigroups in {C}arnot
  groups.
\newblock {\em Discrete Contin. Dyn. Syst. Ser. S}, 11(3):477--491, 2018.

\bibitem{MR2816305}
B.~Galerne.
\newblock Computation of the perimeter of measurable sets via their
  covariogram. {A}pplications to random sets.
\newblock {\em Image Anal. Stereol.}, 30(1):39--51, 2011.

\bibitem{MR3225805}
T.~Grzywny.
\newblock On {H}arnack inequality and {H}\"{o}lder regularity for isotropic
  unimodal {L}\'{e}vy processes.
\newblock {\em Potential Anal.}, 41(1):1--29, 2014.

\bibitem{MR4140542}
T.~Grzywny and K.~Szczypkowski.
\newblock L\'{e}vy processes: concentration function and heat kernel bounds.
\newblock {\em Bernoulli}, 26(4):3191--3223, 2020.

\bibitem{Hardin}
C.~D. Hardin.
\newblock {\em Skewed stable variables and processes}.
\newblock Technical reports of Center for Stochastic Processes UNC, Dept. of
  Statictics, 1984.

\bibitem{MR1287843}
S.~Hiraba.
\newblock Asymptotic behaviour of densities of multi-dimensional stable
  distributions.
\newblock {\em Tsukuba J. Math.}, 18(1):223--246, 1994.

\bibitem{MR3926121}
F.~K\"{u}hn and R.~L. Schilling.
\newblock On the domain of fractional {L}aplacians and related generators of
  {F}eller processes.
\newblock {\em J. Funct. Anal.}, 276(8):2397--2439, 2019.

\bibitem{MR0350027}
N.~S. Landkof.
\newblock {\em Foundations of modern potential theory}.
\newblock Springer-Verlag, New York-Heidelberg, 1972.
\newblock Translated from the Russian by A. P. Doohovskoy, Die Grundlehren der
  mathematischen Wissenschaften, Band 180.

\bibitem{MR3641640}
J.~M. Maz\'{o}n, J.~D. Rossi, and J.~Toledo.
\newblock The heat content for nonlocal diffusion with non-singular kernels.
\newblock {\em Adv. Nonlinear Stud.}, 17(2):255--268, 2017.

\bibitem{MR3930619}
J.~M. Maz\'{o}n, J.~D. Rossi, and J.~J. Toledo.
\newblock {\em Nonlocal perimeter, curvature and minimal surfaces for
  measurable sets}.
\newblock Frontiers in Mathematics. Birkh\"{a}user/Springer, Cham, 2019.

\bibitem{MR710486}
A.~Pazy.
\newblock {\em Semigroups of linear operators and applications to partial
  differential equations}, volume~44 of {\em Applied Mathematical Sciences}.
\newblock Springer-Verlag, New York, 1983.

\bibitem{MR3116054}
M.~van~den Berg.
\newblock Heat flow and perimeter in {$\Bbb{R}^m$}.
\newblock {\em Potential Anal.}, 39(4):369--387, 2013.

\bibitem{MR1262245}
M.~van~den Berg and P.~B. Gilkey.
\newblock Heat content asymptotics of a {R}iemannian manifold with boundary.
\newblock {\em J. Funct. Anal.}, 120(1):48--71, 1994.

\bibitem{MR854867}
V.~M. Zolotarev.
\newblock {\em One-dimensional stable distributions}, volume~65 of {\em
  Translations of Mathematical Monographs}.
\newblock American Mathematical Society, Providence, RI, 1986.
\newblock Translated from the Russian by H. H. McFaden, Translation edited by
  Ben Silver.

\end{thebibliography}

\end{document}